\documentclass[12pt]{article}

\usepackage{amsmath}
\usepackage{amssymb,longtable}
\usepackage[russian,english]{babel}
\usepackage{amsthm}

\newtheorem{definition}{Definition}

\newtheorem{proposition}{Proposition}

\newtheorem{hypothesis}{Conjecture}

\newtheorem{theor}{Theorem}

\theoremstyle{definition}

\begin{document}
\author{A.\,A.~Gornitskii}

\title {Essential Signatures and Canonical Bases for Irreducible Representations of $D_4$.}
\date{}

\maketitle

\thispagestyle{empty}
\begin{abstract}
In the representation theory of simple Lie algebras, we consider the problem of constructing a ``canonical'' weight basis in an arbitrary irreducible finite-dimensional highest weight module. Vinberg suggested a method for constructing such bases by applying the lowering operators corresponding to all positive roots to the highest weight vector. He proposed several conjectures on the parametrization and structure of such bases. It is already known that these conjectures are true for the cases $A_n$, $C_n$, $G_2$, $B_3$. In this paper, we prove these conjectures for the Lie algebra of type $D_4$.
\end{abstract}

\section {Introduction}

Let $\mathfrak{g}$ be a simple Lie algebra.  One has the triangular decomposition
$\mathfrak{g}=\mathfrak{u}^{-}$$\oplus$$\mathfrak{t}$$\oplus$$\mathfrak{u}$, where
$\mathfrak{u}^{-}$ and $\mathfrak{u}$ are mutually opposite maximal unipotent subalgebras, and $\mathfrak{t}$ is a Cartan subalgebra.

One has: $\mathfrak{u}= \langle e_{\alpha}$ $\mid$ $\alpha$ $\in$ $\Delta_{+}\rangle$,
$\mathfrak{u^{-}}=\langle e_{-\alpha}$ $\mid$ $\alpha$ $\in$ $\Delta_{+}\rangle$, where
$\Delta_{+}$ is the system of positive roots, $e_{\pm\alpha}$
 are the root vectors, and the symbol $\langle\ldots\rangle$ stands for the linear span.

 We denote a finite-dimensional irreducible $\mathfrak{g}$-module with highest weight $\lambda$ by $V(\lambda)$ and a highest weight vector in this module by $v_{\lambda}$.

 Various approaches to construction of canonical bases in $V(\lambda)$ are known: Gelfand--Tsetlin bases, crystal bases, etc. For instance, one constructs the crystal basis by applying the lowering operators corresponding to simple roots to the highest weight vector in a certain order; see \cite{[L]}. Vinberg's method \cite{[V]} resembles the method for constructing crystal bases; the difference is that the lowering operators corresponding to \emph{all positive roots} are applied to the highest weight vector. The basic concept used in this method is defined as follows.

\begin{definition}
 A signature is an $(N+1)$-tuple
 $\sigma=(\lambda,p_{{1}},\dots,p_{{N}})$,
  where $N$ is the number of positive roots numbered in a certain fixed order: $\Delta_{+}=\{\alpha_{1},\dots,\alpha_{N}\}$, $\lambda$ is a dominant weight, and $p_{i}\in\mathbb{Z}_{+}$.
\end{definition}

Set
$$
{v}(\sigma)=e_{-\alpha_{1}}^{p_{{1}}}\cdot \ldots \cdot
 e_{-\alpha_{N}}^{p_{{N}}}\cdot{v}_{\lambda}.
$$
$\newline$
$\lambda$ is called the highest weight of $\sigma$ and the weight $\lambda-\sum p_i\alpha_i$ of the vector $v(\sigma)$ is called the weight of $\sigma$. Thus we have defined a vector in $V(\lambda)$ for every signature with highest weight $\lambda$. The vectors ${v}(\sigma)$ generate $V(\lambda)$, but they are linearly dependent. Our goal is to select a basis of $V(\lambda)$ from the set of all vectors ${v}(\sigma)$.

 This problem was solved in \cite{[FFL1]}, \cite{[FFL2]}, \cite{[G]}, \cite{[T]} for the algebras of types $A_n$, $C_{n}$, $G_2$ and $B_3$. Let $\mathfrak{t}_{\mathbb{Z}}\subset\mathfrak{t}$ be the coroot lattice, i.e., the lattice of vectors on which all weights take integer values. The signatures corresponding to the basis vectors are determined by a set of linear inequalities of the form

  \begin{equation}
  \label{7}
  \sum_{j\in M_{i}}a_{ij}p_{j}\leq \lambda(l_{i}),
  \end{equation}
 where $M_{i}\subset
  \{1,\dots,N\}$ are certain subsets, $l_{i}\in \mathfrak{t}_{\mathbb{Z}}$, $a_{ij}=1$ or 2 in type $B_3$, and $a_{ij}=1$ otherwise.

  In this paper, we solve the above problem for the Lie algebra of type $D_4$. The basis vectors for $D_4$ are determined by a set of linear inequalities of the form (\ref{7}),
  where $a_{ij}=1$, 2 or 3 (Proposition \ref{8}).

Having announced the result for $D_4$, we explain our approach to solving the problem in general. To this end, we need to equip the set of signatures with an order. Consider an arbitrary numbering of positive roots. Let us introduce an order on signatures with highest weight $\lambda$. Let $\sigma=(\lambda,p_1,\ldots,p_N)$ and $\sigma'=(\lambda,p_1',\ldots,p_N').$
 Set
 $$
  q_{i}=\sum_{j=1}^{N-i+1}p_{j},
 $$
 $$
  q_{i}'=\sum_{j=1}^{N-i+1}p_{j}'.
 $$
  Then $\sigma<\sigma'$ if and only if $(q_{1},\ldots,q_{N})<(q_{1}',\ldots,q_{N}')$ in the lexicographic order.

 \begin{definition} A signature $\sigma$ is essential, if
 $v(\sigma)\notin\langle v(\tau)\mid\tau<\sigma\rangle$.

\end{definition}
The following fact is obvious.
\begin{proposition}
The set $\{{v}(\sigma)\mid\sigma \mbox{ essential}\}$ is a basis of $V(\lambda)$.
\end{proposition}

 The essential signatures with given highest weight $\lambda$ parametrize the desired canonical basis of $V(\lambda)$. The following proposition was proved by Vinberg. For convenience of
the reader, we provide a proof in Section \ref{2}:

 \begin{proposition}
 The essential\label{1} signatures form a semigroup in $\mathfrak{t}^{*}_{\mathbb{Z}}\oplus\mathbb{Z}^{N}$.

\end{proposition}

Now we proceed to the first conjecture of Vinberg about the structure of the set of essential signatures.
\begin{hypothesis} The semigroup of essential signatures is generated by the essential signatures of fundamental highest weights.
 \end{hypothesis}

Let us formulate other conjectures of Vinberg. Let
 $\Sigma \subset \mathfrak{t}_{\mathbb{Z}}^{*}\oplus \mathbb{Z}^{N}$ be the semigroup of essential signatures
   and let $\Sigma_{\mathbb{Q}}$ be the rational cone spanned by $\Sigma$. Then this cone can be defined by linear inequalities. (The number of these inequalities is finite if Conjecture 1 holds.)
\begin{hypothesis}
 The semigroup $\Sigma$ is saturated, i.e.,
 $\Sigma=\Sigma_{\mathbb{Q}}\bigcap (\mathfrak{t}_{\mathbb{Z}}^{*}\oplus
 \mathbb{Z}^{N})$.
\end{hypothesis}
Conjecture 2 claims that the bases of $V(\lambda)$ are parametrized by lattice points of plane sections of some polyhedral cone.
\begin{hypothesis}
 There exist a family of subsets $M_{i}\subset
  \{1,\dots,N\}$ and a family of elements $l_{i}\in \mathfrak{t}_{\mathbb{Z}}$ such that the set of essential signatures $\sigma=(\lambda,p_{1},\dots,p_{N})$ of highest weight $\lambda$ is given by the inequalities

  $$
  \sum_{j\in M_{i}}p_{j}\leq \lambda(l_{i}).
  $$

\end{hypothesis}
Conjecture 3 refines the structure of the polyhedral cone in Conjecture 2.

As we already noted, in the cases $B_3$ and $D_4$, the set of essential signatures of highest weight $\lambda$ for the chosen numbering of positive roots is given by the inequalities of the form

  $$
  \sum_{j\in M_{i}}a_{ij}p_{j}\leq \lambda(l_{i}),
  $$
i.e., a modified version of Conjecture 3 holds.

To prove the conjectures for $D_4$ we use the same method as in \cite{[T]}. However we have a general approach which can be used to prove the conjectures for all Lie algebras of types $A_n, C_n, B_3, D_4, G_2$. This will be published elsewhere.

Let $\Sigma^{f}$ be the semigroup generated by the set of essential signatures of fundamental highest weights, $\Sigma^{f}_{\mathbb{Q}}$ be the rational cone spanned by $\Sigma^{f}$ and $\Sigma^{f}(\lambda)$ be the set of signatures $\sigma$ of highest weight $\lambda$ such that $\sigma\in\Sigma^{f}_{\mathbb{Q}}$.

In Section \ref{Sec3} we fix some numbering of positive roots of $D_4$ and find the inequalities defining the cone $\Sigma^{f}_{\mathbb{Q}}$ (see Proposition \ref{4}). Then we prove:

\begin{theor}
\label{t1}
$\Sigma^{f}(\lambda+\mu)=\Sigma^{f}(\lambda)+\Sigma^{f}(\mu).$
\end{theor}

Finally, we prove:
\begin{theor}
\label{t2}
For an arbitrary dominant weight $\lambda$,
$$|\Sigma^{f}(\lambda)|=\dim V(\lambda).$$
\end{theor}

Theorems 1, 2 together with Proposition \ref{1} imply that $\Sigma=\Sigma^{f}$ and that Conjectures 1, 2 and a modified version of Conjecture 3 are true for the Lie algebra of type $D_4$.

\section{The semigroup of essential signatures}
\label{2}

Now we show that the essential signatures of all highest weights form a semigroup.

Let $G$ be a simply connected simple complex algebraic group such that $\mathop{\mathrm{Lie}}G=\mathfrak{g}$. Let $T$ be the maximal torus in $G$ such that $\mathop{\mathrm{Lie}}T=\mathfrak{t}$ and $U$ be the maximal unipotent subgroup of $G$ such that
 $\mathop{\mathrm{Lie}}U=\mathfrak{u}$.
 Consider the homogeneous space $G/U$.
 Let $B=T\rightthreetimes U$ be the Borel subgroup. Then

$$
 \mathbb{C}[G/U]=\bigoplus_{\lambda} \mathbb{C}[G]_{\lambda}^{(B)},
$$
where $$\mathbb{C}[G]_{\lambda}^{(B)}=\{f\in\mathbb{C}[G]\mid f(gtu)=\lambda(t)f(g),\, \forall g\in G, t\in T, u\in U\}$$
is the subspace of eigenfunctions of weight $\lambda$ for $B$ acting on $\mathbb{C}[G]$ by right translations of an argument.
 Each subspace $\mathbb{C}[G]_{\lambda}^{(B)}$ is finite-dimensional and is isomorphic as a $G$-module (with respect to the action of $G$ by left translations of an argument), to the space $V(\lambda)^{*}$ of linear functions on $V(\lambda)$ (see \cite{[П]}, Theorem 3). The isomorphism is given by the formula:
$$
 V(\lambda)^{*}\ni\omega \longmapsto f_{\omega}\in\mathbb{C}[G]_{\lambda}^{(B)},
 \quad\textrm{where}\quad f_{\omega}(g)=\langle\omega, g\emph{v}_{\lambda}\rangle.
$$

Let $U^{-}$ be the maximal unipotent subgroup such that $\mathop{\mathrm{Lie}}U^{-}=\mathfrak{u^{-}}$. The function $f_{\omega}$ is uniquely determined by its restriction to the dense open subset $U^{-}\cdot$$T\cdot$$U$; moreover
\begin{multline*}
 $$f_{\omega}(u^{-}\cdot t\cdot u)=\langle\omega
,u^{-} tu\emph{v}_{\lambda}\rangle=\langle\omega
,\lambda(t)u^{-}\emph{v}_{\lambda}\rangle=\lambda(t)f_{\omega}(u^{-}),\\ \quad \forall u\in U, u^{-}\in U^{-}, t\in T.
$$
\end{multline*}
Next, $U^{-}=U_{-\alpha_{1}}\cdot\ldots\cdot U_{-\alpha_{N}}$, where
$U_{\alpha}=\{\exp(te_{\alpha})\mid$ $t\in\mathbb{C}\}$ (see \cite[Sec. X, \S 28.1]{[X]}). Hence

$$
 u^{-}=\exp(t_{1}e_{-\alpha_{1}})\cdot\ldots\cdot\exp(t_{N}e_{-\alpha_{N}}).
$$
Thus we obtain
$$
f_{\omega}(u^{-})=\left\langle\omega,
\exp(t_{1}e_{-\alpha_{1}})\cdot\ldots\cdot\exp(t_{N}e_{-\alpha_{N}})
\cdot \emph{v}_{\lambda}\right\rangle=\sum_{\sigma=(\lambda,p_{1},\dots,p_{N})}
\frac{\prod t_{i}^{p_{i}}}{\prod p_{i}!}
\langle\omega,\emph{v}(\sigma)\rangle.
$$

\begin{proposition}
 A signature $\sigma$ is essential if and only if
$\prod t_{i}^{p_{i}}$ is the least term in
 $f_{\omega}|_{U^{-}}$ for some
$\omega \in V(\lambda)^{*}$ in the sense of the order introduced above.
\end{proposition}
\begin{proof}
Let $\prod t_{i}^{p_{i}}$ be the least term in $f_{\omega}|_{U^{-}}$ for some $\omega\in V(\lambda)^{*}$. Then $\omega$ vanishes on all vectors $\emph{v}(\tau)$ with $\tau<\sigma$ and is nonzero at  $\emph{v}(\sigma)$. Consequently, $\emph{v}(\sigma)$ cannot be expressed via $\emph{v}(\tau)$ with $\tau<\sigma$, and hence  $\sigma$ is essential.

Conversely, let $\sigma$ be essential. Consider a function $\omega$ that vanishes on $\emph{v}(\tau)$ for all essential $\tau$ except for $\sigma$. Obviously, $f_{\omega}|_{U^{-}}$ has the desired least term.
\end{proof}

\begin{proof}[Proof of Proposition \ref{1}]
Suppose that the least terms in $f|_{U^{-}}$ and $g|_{U^{-}}$ correspond to the essential signatures $\lambda$ and $\mu$. Then the least term in $(f\cdot g)|_{U^{-}}$ corresponds to the signature $\lambda+\mu$. Hence $\lambda+\mu$ is essential.
\end{proof}

\section {Essential signatures in type $D_4$}
\label{Sec3}
Consider the case $D_4$. Let $\beta_1, \beta_2, \beta_3, \beta_4$ be the set of simple roots for $D_4$ and let $\omega_1, \omega_2, \omega_3, \omega_4$ be the fundamental weights:
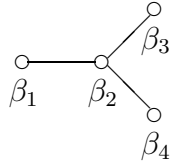
\begin{center}
\begin{picture}(100,80)
\put(10,50){\circle{5}}
\put(12,50){\line(1,0){26}}
\put(5,35){$\beta_1$}
\put(40,50){\circle{5}}
\put(35,35){$\beta_2$}
\put(42.5,51.5){\line(1,1){16}}
\put(42.5,48.5){\line(1,-1){16}}
\put(60,70){\circle{5}}
\put(55,55){$\beta_3$}
\put(60,30){\circle{5}}
\put(55,15){$\beta_4$}
\end{picture}
\end{center}

Denote the weights of the representation $V(\omega_1)$ by $\pm\varepsilon_1, \pm\varepsilon_2, \pm\varepsilon_3, \pm\varepsilon_4$. One has:

\begin{tabular}{llll}
$\beta_1=\varepsilon_1-\varepsilon_2$&
$\beta_2=\varepsilon_2-\varepsilon_3$&
$\beta_3=\varepsilon_3-\varepsilon_4$&
$\beta_4=\varepsilon_3+\varepsilon_4$.
\end{tabular}

 Let us number the positive roots as follows:

\begin{center}
\begin{tabular}{lll}
$\alpha_1=\varepsilon_1+\varepsilon_2$ &
$\alpha_2=\varepsilon_2+\varepsilon_3$ &
$\alpha_3=\varepsilon_1+\varepsilon_3$ \\
$\alpha_4=\varepsilon_3+\varepsilon_4$&
$\alpha_5=\varepsilon_2+\varepsilon_4$ &
$\alpha_6=\varepsilon_1+\varepsilon_4$ \\
$\alpha_7=\varepsilon_1-\varepsilon_3$ &
$\alpha_8=\varepsilon_1-\varepsilon_4$&
$\alpha_9=\varepsilon_3-\varepsilon_4$ \\
$\alpha_{10}=\varepsilon_2-\varepsilon_4$ &
$\alpha_{11}=\varepsilon_2-\varepsilon_3$ &
$\alpha_{12}=\varepsilon_1-\varepsilon_2$.
\end{tabular}
\end{center}

Now we need to obtain all essential signatures of fundamental highest weights.

 The representations of highest weights $\omega_1, \omega_3, \omega_4$ have one-dimensional weight subspaces. Hence it is easy to find essential signatures. Indeed, to obtain the essential signature of the weight $\mu$ we just need to find the minimal signature $\sigma$ such that the vector $v(\sigma)$ has weight $\mu$.

 Here are the essential signatures of highest weight $\omega_1$ (the highest weight component is omitted):

\begin{center}
\begin{tabular}{lll}
1. $(0,0,0,0,0,0,0,0,0,0,0,0)$  & 2. $(0,0,0,0,0,0,0,0,0,0,0,1)$ \\ 3. $(0,0,0,0,0,0,0,1,0,0,0,0)$  &4. $(0,0,0,0,0,0,1,0,0,0,0,0)$ \\

5. $(0,0,0,0,0,1,0,0,0,0,0,0)$ &6. $(0,0,1,0,0,0,0,0,0,0,0,0)$ \\7. $(1,0,0,0,0,0,0,0,0,0,0,0)$ & 8. $(1,0,0,0,0,0,0,0,0,0,0,1)$.
\end{tabular}
\end{center}
Here are the essential signatures of highest weight $\omega_3$ (the highest weight component is omitted):

\begin{center}
\begin{tabular}{lll}
1. $(0,0,0,0,0,0,0,0,0,0,0,0)$  & 2. $(0,0,0,0,0,0,0,0,0,1,0,0)$ \\ 3. $(0,0,0,0,0,0,0,0,1,0,0,0)$  &4. $(0,0,0,0,0,0,0,1,0,0,0,0)$ \\

5. $(0,0,1,0,0,0,0,0,0,0,0,0)$ &6. $(0,1,0,0,0,0,0,0,0,0,0,0)$ \\7. $(1,0,0,0,0,0,0,0,0,0,0,0)$& 8. $(0,0,1,0,0,0,0,0,0,1,0,0)$.
\end{tabular}
\end{center}
Here are the essential signatures of highest weight $\omega_4$ (the highest weight component is omitted):

\begin{center}
\begin{tabular}{ll}
1. $(0,0,0,0,0,0,0,0,0,0,0,0)$  & 2. $(0,0,0,0,0,1,0,0,0,0,0,0)$ \\ 3. $(0,0,0,0,1,0,0,0,0,0,0,0)$ &4. $(0,0,0,1,0,0,0,0,0,0,0,0)$\\

5. $(0,0,1,0,0,0,0,0,0,0,0,0)$ &6. $(0,1,0,0,0,0,0,0,0,0,0,0)$ \\7. $(1,0,0,0,0,0,0,0,0,0,0,0)$& 8. $(0,1,0,0,0,1,0,0,0,0,0,0)$.
\end{tabular}
\end{center}

The representation of highest weight $\omega_2$ is the adjoint representation. If $\mu\neq0$ is a weight of the representation $V(\omega_2)$, then the weight subspace $ V_{\mu}$ is one-dimensional, and $\dim V_{0}=4$. It is easy to verify that the vectors $v(\sigma_i)$, where $\sigma_i$ are the four minimal signatures of weight 0, ($i=1, 2, 3, 4$), are linearly independent.
Here are the essential signatures of highest weight $\omega_2$ (the highest weight component is omitted):

\begin{center}
\begin{tabular}{ll}
1. $(0,0,0,0,0,0,0,0,0,0,0,0)$  & 2. $(0,0,0,0,0,0,0,0,0,0,1,0)$ \\ 3. $(0,0,0,0,0,0,0,0,0,1,0,0)$ & 4. $(0,0,0,0,0,0,0,1,0,0,0,0)$ \\
5. $(0,0,0,0,0,0,1,0,0,0,0,0)$ &6. $(0,0,0,0,0,1,0,0,0,0,0,0)$ \\7. $(0,0,0,0,1,0,0,0,0,0,0,0)$& 8. $(0,0,1,0,0,0,0,0,0,0,0,0)$\\
9. $(0,1,0,0,0,0,0,0,0,0,0,0)$ & 10. $(1,0,0,0,0,0,0,0,0,0,0,0)$ \\11. $(0,0,0,0,0,0,0,1,0,0,1,0)$ & 12. $(0,0,0,0,0,1,0,0,0,0,1,0)$\\
13. $(0,0,1,0,0,0,0,0,0,0,1,0)$ & 14. $(0,1,0,0,0,0,0,0,0,0,1,0)$ \\15. $(1,0,0,0,0,0,0,0,0,0,1,0)$ & 16. $(0,0,0,0,0,1,0,0,0,1,0,0)$\\
17. $(0,0,1,0,0,0,0,0,0,1,0,0)$ & 18. $(1,0,0,0,0,0,0,0,0,1,0,0)$ \\ 19. $(0,0,0,0,0,1,0,1,0,0,0,0)$ & 20. $(0,0,0,0,1,0,0,1,0,0,0,0)$ \\
21. $(1,0,0,0,0,0,0,1,0,0,0,0)$ & 22. $(1,0,0,0,0,0,1,0,0,0,0,0)$ \\
23. $(0,1,0,0,0,1,0,0,0,0,0,0)$ & 24. $(1,0,0,0,0,1,0,0,0,0,0,0)$\\
25. $(1,0,0,0,1,0,0,0,0,0,0,0)$ & 26. $(1,0,1,0,0,0,0,0,0,0,0,0)$ \\ 27. $(1,1,0,0,0,0,0,0,0,0,0,0)$ & 28. $(2,0,0,0,0,0,0,0,0,0,0,0)$.
\end{tabular}
\end{center}

 Recall that $\Sigma^{f}$ is the semigroup generated be the essential signatures of fundamental highest weights and $\Sigma^{f}_{\mathbb{Q}}$ is the rational cone spanned by $\Sigma^{f}$. We use the
following notation for the coordinates of a dominant weight in the basis
of fundamental weights:   $\lambda=k_1\omega_1+k_2\omega_2+k_3\omega_3+k_4\omega_4$.
\begin{proposition}
The cone $\Sigma^{f}_{\mathbb{Q}}$\label{4} is given by the inequalities:
\label{8}
\begin{enumerate}
{\scriptsize
\item $p_{12}\leq k_1$
\item $p_{11}\leq k_2$
\item $p_9\leq k_3$
\item $p_4\leq k_4$
\item $p_7+p_{11}+p_{12}\leq k_1+k_2$
\item $p_7+p_8+p_9+p_{10}+p_{12}\leq k_1+k_2+k_3$
\item $p_7+p_9+p_{10}+p_{11}+p_{12}\leq k_1+k_2+k_3$
\item $p_9+p_{10}+p_{11}\leq k_2+k_3$
\item $p_4+p_5+p_6+p_7+p_{12}\leq k_1+k_2+k_4$
\item $p_4+p_5+p_7+p_{11}+p_{12}\leq k_1+k_2+k_4$
\item $p_4+p_5+p_{11}\leq k_2+k_4$
\item $p_2+p_4+p_5+p_9+p_{10}\leq k_2+k_3+k_4$
\item $p_4+p_5+p_9+p_{10}+p_{11}\leq k_2+k_3+k_4$
\item $p_3+p_4+p_5+p_6+p_7+p_9+p_{12}\leq k_1+k_2+k_3+k_4$
\item $p_2+p_3+p_4+p_5+p_7+p_9+p_{12}\leq k_1+k_2+k_3+k_4$
\item $p_2+p_3+p_4+p_7+p_8+p_9+p_{12}\leq k_1+k_2+k_3+k_4$
\item $p_2+p_4+p_7+p_8+p_9+p_{10}+p_{12}\leq k_1+k_2+k_3+k_4$
\item $p_4+p_5+p_7+p_9+p_{10}+p_{11}+p_{12}\leq k_1+k_2+k_3+k_4$
\item $p_2+p_4+p_5+p_7+p_9+p_{10}+p_{12}\leq k_1+k_2+k_3+k_4$
\item $p_1+p_3+p_4+p_5+p_6+p_7+p_8+p_9+p_{11}\leq k_1+2k_2+k_3+k_4$
\item $p_1+p_2+p_3+p_4+p_5+p_7+p_8+p_9+p_{11}\leq k_1+2k_2+k_3+k_4$
\item $p_3+p_4+p_5+p_6+p_7+p_8+p_9+p_{11}+p_{12}\leq k_1+2k_2+k_3+k_4$
\item $p_2+p_3+p_4+p_5+p_7+p_8+p_9+p_{11}+p_{12}\leq k_1+2k_2+k_3+k_4$
\item $p_1+p_2+p_4+p_5+p_7+p_8+p_9+p_{10}+p_{11}\leq k_1+2k_2+k_3+k_4$
\item $p_2+p_4+p_5+p_7+p_8+p_9+p_{10}+p_{11}+p_{12}\leq k_1+2k_2+k_3+k_4$
\item $p_1+p_4+p_5+p_6+p_7+p_8+p_9+p_{10}+p_{11}\leq k_1+2k_2+k_3+k_4$
\item $p_4+p_5+p_6+p_7+p_8+p_9+p_{10}+p_{11}+p_{12}\leq k_1+2k_2+k_3+k_4$
\item $p_2+p_3+p_4+p_5+p_7+p_8+2p_9+p_{10}+p_{11}+p_{12}\leq k_1+2k_2+2k_3+k_4$
\item $p_3+p_4+p_5+p_6+p_7+p_8+2p_9+p_{10}+p_{11}+p_{12}\leq k_1+2k_2+2k_3+k_4$
\item $p_1+p_2+p_3+p_4+p_5+p_7+p_8+2p_9+p_{10}+p_{11}\leq k_1+2k_2+2k_3+k_4$
\item $p_1+p_3+p_4+p_5+p_6+p_7+p_8+2p_9+p_{10}+p_{11}\leq k_1+2k_2+2k_3+k_4$
\item $p_2+p_4+p_5+p_7+p_8+2p_9+2p_{10}+p_{11}+p_{12}\leq k_1+2k_2+2k_3+k_4$
\item $p_2+p_3+2p_4+p_5+p_6+p_7+p_8+p_9+p_{11}+p_{12}\leq k_1+2k_2+k_3+2k_4$
\item $p_2+2p_4+p_5+p_6+p_7+p_8+p_9+p_{10}+p_{11}+p_{12}\leq k_1+2k_2+k_3+2k_4$
\item $p_1+p_2+p_3+2p_4+p_5+p_6+p_7+p_8+p_9+p_{11}\leq k_1+2k_2+k_3+2k_4$
\item $p_1+p_2+2p_4+p_5+p_6+p_7+p_8+p_9+p_{10}+p_{11}\leq k_1+2k_2+k_3+2k_4$
\item $p_2+p_3+2p_4+2p_5+p_6+p_7+p_9+p_{11}+p_{12}\leq k_1+2k_2+k_3+2k_4$
\item $p_2+2p_4+2p_5+p_6+p_7+p_9+p_{10}+p_{11}+p_{12}\leq k_1+2k_2+k_3+2k_4$
\item $p_3+p_4+p_5+p_6+2p_7+p_8+p_9+p_{11}+2p_{12}\leq 2k_1+2k_2+k_3+k_4$
\item $p_4+p_5+p_6+2p_7+p_8+p_9+p_{10}+p_{11}+2p_{12}\leq 2k_1+2k_2+k_3+k_4$
\item $p_2+p_3+p_4+p_5+2p_7+p_8+p_9+p_{11}+2p_{12}\leq 2k_1+2k_2+k_3+k_4$
\item $p_2+p_4+p_5+2p_7+p_8+p_9+p_{10}+p_{11}+2p_{12}\leq 2k_1+2k_2+k_3+k_4$
\item$p_1+p_2+p_3+2p_4+p_5+p_6+p_7+p_8+2p_9+p_{10}+p_{11}\leq k_1+2k_2+2k_3+2k_4$
\item $p_2+p_3+2p_4+p_5+p_6+p_7+p_8+2p_9+p_{10}+p_{11}+p_{12}\leq k_1+2k_2+2k_3+2k_4$
\item $p_2+p_3+2p_4+2p_5+p_6+p_7+2p_9+p_{10}+p_{11}+p_{12}\leq k_1+2k_2+2k_3+2k_4$
\item $p_2+p_3+2p_4+p_5+p_6+2p_7+p_8+p_9+p_{11}+2p_{12}\leq 2k_1+2k_2+k_3+2k_4$
\item $p_2+2p_4+p_5+p_6+2p_7+p_8+p_9+p_{10}+p_{11}+2p_{12}\leq 2k_1+2k_2+k_3+2k_4$
\item $p_2+p_3+2p_4+2p_5+p_6+2p_7+p_9+p_{11}+2p_{12}\leq 2k_1+2k_2+k_3+2k_4$
\item $p_2+2p_4+2p_5+p_6+2p_7+p_9+p_{10}+p_{11}+2p_{12}\leq 2k_1+2k_2+k_3+2k_4$
\item $p_2+p_3+p_4+p_5+2p_7+p_8+2p_9+p_{10}+p_{11}+2p_{12}\leq 2k_1+2k_2+2k_3+k_4$
\item $p_3+p_4+p_5+p_6+2p_7+p_8+2p_9+p_{10}+p_{11}+2p_{12}\leq 2k_1+2k_2+2k_3+k_4$
\item $p_2+p_4+p_5+2p_7+p_8+2p_9+2p_{10}+p_{11}+2p_{12}\leq 2k_1+2k_2+2k_3+k_4$
\item  $p_2+p_3+2p_4+p_5+p_6+2p_7+p_8+2p_9+p_{10}+p_{11}+2p_{12}\leq 2k_1+2k_2+2k_3+2k_4$
\item $p_2+p_3+2p_4+2p_5+p_6+2p_7+2p_9+p_{10}+p_{11}+2p_{12}\leq 2k_1+2k_2+2k_3+2k_4$
\item $p_2+p_3+2p_4+2p_5+p_6+p_7+p_8+p_9+2p_{11}+p_{12}\leq k_1+3k_2+k_3+2k_4$
\item $p_2+2p_4+2p_5+p_6+p_7+p_8+p_9+p_{10}+2p_{11}+p_{12}\leq k_1+3k_2+k_3+2k_4$
\item $p_1+p_2+p_3+2p_4+2p_5+p_6+p_7+p_8+p_9+2p_{11}\leq k_1+3k_2+k_3+2k_4$
\item $p_1+p_2+2p_4+2p_5+p_6+p_7+p_8+p_9+p_{10}+2p_{11}\leq k_1+3k_2+k_3+2k_4$
\item $p_1+p_2+2p_4+2p_5+p_6+p_7+p_8+2p_9+2p_{10}+2p_{11}\leq k_1+3k_2+2k_3+2k_4$
\item $p_2+2p_4+2p_5+p_6+p_7+p_8+2p_9+2p_{10}+2p_{11}+p_{12}\leq k_1+3k_2+2k_3+2k_4$
\item $p_1+p_2+2p_4+2p_5+p_6+2p_7+p_8+p_9+p_{10}+2p_{11}+p_{12}\leq 2k_1+3k_2+k_3+2k_4$
\item $p_1+p_2+p_3+2p_4+2p_5+p_6+2p_7+p_8+p_9+2p_{11}+p_{12}\leq 2k_1+3k_2+k_3+2k_4$
\item $p_1+p_2+2p_4+2p_5+p_6+2p_7+p_8+2p_9+2p_{10}+2p_{11}+p_{12}\leq 2k_1+3k_2+2k_3+2k_4$
\item $p_2+p_3+2p_4+2p_5+p_6+p_7+p_8+3p_9+2p_{10}+2p_{11}+p_{12}\leq k_1+3k_2+3k_3+2k_4$
\item $p_1+p_2+p_3+2p_4+2p_5+p_6+p_7+p_8+3p_9+2p_{10}+2p_{11}\leq k_1+3k_2+3k_3+2k_4$
\item $p_2+p_3+2p_4+2p_5+p_6+3p_7+p_8+p_9+2p_{11}+3p_{12}\leq 3k_1+3k_2+k_3+2k_4$
\item $p_2+2p_4+2p_5+p_6+3p_7+p_8+p_9+p_{10}+2p_{11}+3p_{12}\leq 3k_1+3k_2+k_3+2k_4$
\item
$p_2+p_3+2p_4+2p_5+p_6+3p_7+p_8+3p_9+2p_{10}+2p_{11}+3p_{12}\leq 3k_1+3k_2+3k_3+2k_4$
\item
$p_1+p_2+p_3+2p_4+2p_5+p_6+2p_7+p_8+3p_9+2p_{10}+2p_{11}+p_{12}\leq 2k_1+3k_2+3k_3+2k_4$
\item
$p_2+2p_4+2p_5+p_6+3p_7+p_8+2p_9+2p_{10}+2p_{11}+3p_{12}\leq 3k_1+3k_2+2k_3+2k_4$.

\par}
\end{enumerate}

\end{proposition}
\begin{proof}
We have found all essential signatures of fundamental highest weights. Hence we have the inequalities defining the dual cone $(\Sigma^{f}_{\mathbb{Q}})^{\vee}$. (A covector $v$ is in $(\Sigma^{f}_{\mathbb{Q}})^{\vee}$ if and only if the pairing $\langle v,\sigma_i\rangle$ is nonnegative for each essential signature $\sigma_i$ of fundamental highest weight.) To obtain the inequalities defining the cone $\Sigma^{f}_{\mathbb{Q}}$ we need to find generators of the dual cone $(\Sigma^{f}_{\mathbb{Q}})^{\vee}$. It can be done by using a standard algorithm for finding generators of a cone defined by linear inequalities in $d$-dimensional vector space. We select $d-1$ inequalities and turn them into equalities, thus obtaining a system of linear equations. Of all these systems, we select systems of maximal rank (i.e., systems with one-dimensional space of solutions). For each system of maximal rank, we check whether a non-zero solution of it or the opposite vector is in the cone. These vectors are the desired generators.
\end{proof}
Recall that we denote a set of signatures of highest weight $\lambda$ satisfying the above inequalities by $\Sigma^{f}(\lambda)$.

\begin{proof}[Proof of Theorem \ref{t1}]
Let $\sigma=(\lambda,p_1,\ldots,p_{12})\in \Sigma^{f}(\lambda)$ and let $\lambda=k_1\omega_1+k_2\omega_2+k_3\omega_3+k_4\omega_4$. It is enough to find an essential signature $\tau=(\omega_i,q_1,\ldots,q_{12})$ of highest weight $\omega_i$ such that $\sigma-\tau\in \Sigma^{f}(\lambda-\omega_i)$.

To prove the Theorem we find some essential signatures $\tau$ of fundamental highest weights such that either $\sigma-\tau\in\Sigma^{f}(\lambda-\omega_i)$ or the conditions of the form $p_s=0$ holds. Arguing like that, we reduce the problem to the case where $\sigma$ has a simple structure. In this case we show directly that $\sigma$ is representable as a sum of essential signatures of fundamental highest weights.

To verify that $\sigma-\tau\in\Sigma^{f}(\lambda-\omega_i)$ we have to check that, after subtracting $\tau$, the left-hand side of each inequality decreases not less than right-hand side. Note that if we suppose $p_s=0$ for some $s=1,\ldots,12$ then we need to check condition above only for essential inequalities. For example, if $p_i=0$ for all $i\neq12$, then we have only one inequality (the first inequality) to check.

We use a notation $(\omega_i,j,k)$ for the essential signature $(\omega_i,s_1,\ldots,s_{12})$ of highest weight $\omega_i$ with $s_j=s_k=1$ and $s_l=0$ for $l\neq k,j$; $(\omega_i,j)$ stands for the essential signature of highest weight $\omega_i$ with $s_j=1$ and $s_l=0$ for $l\neq j$.

\paragraph{\textbf{Case} $\bf{k_1, k_2, k_3, k_4}>\textbf{0}$.}
$\newline$
First we want to show that we can assume $p_{12}=0$.
$\newline$
Suppose $p_1\neq0$ and $p_{12}\neq0$, then $\sigma-\tau\in\Sigma^{f}(\lambda-\omega_1)$ for $\tau=(\omega_1,1,{12})$. Hence $p_1=0$ or $p_{12}=0$. If $p_{12}=k_1$ and $p_1=0$, then we can take $\tau=(\omega_1,{12})$. Hence we can assume $p_1=0, p_{12}<k_1$.
 $\newline$
 Next, if $p_7\neq0$, then take $\tau=(\omega_1,7)$, hence we can suppose $p_7=0, p_{12}<k_1$ and $p_1=0$. In this case if $p_{12}\neq0$, then $\tau=(\omega_1,{12})$. Thus we obtain $\bf{p_{12}=0}$.

 $\newline$
 If $p_7\neq0$, then take $\tau=(\omega_1,7)$ hence we can assume $\bf{p_7=0}$.
 $\newline$
 If $p_9\neq0$, then take $\tau=(\omega_3,9)$ hence we can assume $\bf{p_9=0}$.
 $\newline$
 If $p_4\neq0$, then take $\tau=(\omega_4,4)$ hence we can assume $\bf{p_4=0}$.

$\newline$
Now we show that we can suppose $p_{10}=0$. If $p_3\neq0$ and $p_{10}\neq0$, then take $\tau=(\omega_3,3,{10})$ hence we can assume $p_3=0$ or $p_{10}=0$. If $p_3=0$, then take $\tau=(\omega_3,{10})$ hence we obtain $p_{10}=0$. Thus we can assume $\bf{p_{10}=0}$.

$\newline$
Now we show that we can assume $p_5=0$. If $p_{11}=k_2$, then take $\tau=(\omega_4,5)$ hence we get $p_5=0$. If $p_{11}\neq k_2$, then take $\tau=(\omega_2,5,8)$ hence we can assume $p_5=0$ or $p_8=0$.
$\newline$
Suppose $p_8=0$. Then take $\tau=(\omega_4,2,6)$ hence we obtain $p_2=0$ or $p_6=0$ or the inequality 11 turns into equality for $\sigma$. In any case we can take $\tau=(\omega_4,5)$ hence we get $p_5=0$ (the inequalities 35, 11 can not turn into equalities together, because otherwise the inequality 57 is violated). So $\bf{p_5=0}$.

$\newline$
Now we show that $p_6=0$.
If $p_2\neq0$ and $p_6\neq0$, then take $\tau=(\omega_4,2,6)$. Then $p_2=0$ or $p_6=0$. Suppose $p_2=0$, $p_6\neq0$, then take $\tau=(\omega_4,6)$ hence we obtain $p_6=0$ or the inequality 16 turns into equality for $\sigma$. If $p_8=0$, then take $\tau=(\omega_4,6)$ hence we can suppose $p_6=0$. If $p_8\neq0$, then take $\tau=(\omega_3,8)$ hence we get $p_8=0$ or the inequality 14 turns into equality. If both inequalities 14, 16 turn into equalities, then take $\tau=(\omega_2,6,8)$ hence we obtain $p_6=0$ or $p_8=0$. Thus we can assume $\bf{p_6=0}$.

 $\newline$
 If $p_2\neq0$, then take $\tau=(\omega_4,2)$ hence we get $\bf{p_2=0}$.
  $\newline$
  Thus only the coordinates $p_1, p_3, p_8, p_{11}$ can be nonzero. Hence $\sigma$ is representable as a sum of essential signatures of fundamental highest weights (essential inequalities are 2, 6, 16, 21). Indeed, we have the following inequalities:
  \begin{enumerate}
  \item $p_{11}\leq k_2$
  \item $p_8\leq k_1+k_2+k_3$
  \item $p_3+p_8\leq k_1+k_2+k_3+k_4$
  \item $p_1+p_3+p_8+p_{11}\leq k_1+2k_2+k_3+k_4.$
  \end{enumerate}
  If $p_8\neq0$, then take $\tau=(\omega_3,8)$ hence we can assume $p_8=0$. If $p_3\neq0$, then take $\tau=(\omega_1,3)$ hence we obtain $p_3=0$. If $p_{11}\neq0$ and $p_1\neq0$, then take $\tau=(\omega_2,p_1,p_{11})$ hence we get $p_{11}=0$ or $p_1=0$. Finally, take $\tau=(\omega_2,1)$ or $\tau=(\omega_2,p_{11})$.

   Thus we reduced $\sigma$ to zero.

$\newline$
We have considered the case $k_1, k_2, k_3, k_4>0$. For arbitrary $k_1,k_2,k_3,k_4$ we still may assume $\bf{p_{12}=0}, \bf{p_4=0}, \bf{p_9=0}$ using the same arguments as above. Moreover, if $k_1\neq0$, then take $\tau=(\omega_1,7)$ hence we get $p_7=0$. If $k_1=0, k_2\neq0,p_1\neq0$, then take $\tau=(\omega_2,1,7)$ hence we obtain $p_7=0$ or $p_1=0$. If $k_1=0, k_2\neq0, p_1=0$, then take $\tau=(\omega_2,7)$ hence we get $p_7=0$. Finally, if $k_1=0, k_2=0$, then $p_7=0$ (inequality 5). Thus we can assume $\bf{p_7=0}$ in any case.
 $\newline$
 If $k_3\neq0$, then take $\tau=(\omega_3,3,{10})$ hence we obtain $p_3=0$ or $p_{10}=0$. If $k_3=0, k_2\neq0$, then take $\tau=(\omega_2,3,{10})$ $p_3=0$ or $p_{10}=0$. If  $k_3=0, k_2=0$, then $p_{10}=0$ (inequality 8). Thus we can assume $p_3=0$ or $p_{10}=0$ in any case.

 \paragraph{\textbf{Case} $\bf{k_1, k_2, k_4}>\bf{0, k_3=0}$.}
  $\newline$
  From the arguments above we get $p_3=0$ or $p_{10}=0$. We want to show that we can suppose $p_{10}=0$. Suppose $p_3=0$.
  $\newline$
  If $p_2\neq0$ and $p_6\neq0$, then take $\tau=(\omega_4,2,6)$ hence we obtain $p_2=0$ or $p_6=0$ or the inequality 13 turns into equality. In the last case necessarily $p_5\neq0$ (see inequality 8) and then $\tau=(\omega_4,5)$. Hence $p_2=0$ or $p_6=0$.
   $\newline$
   If $p_6=0$ and $p_1\neq0$, then take $\tau=(\omega_2,1,{10})$ hence we get $p_1=0$ or $p_{10}=0$. If $p_6=0, p_1=0$, then take $\tau=(\omega_2,{10})$ hence we can assume $p_{10}=0$. Thus $p_{10}=0$ in the case $p_6=0$.
    $\newline$
    If $p_2=0$, $p_6\neq0$, then take $\tau=(\omega_2,6,{10})$ hence we obtain $p_6=0$ or $p_{10}=0$. Thus we can assume $\bf{p_{10}=0}$.

    Now we can use the arguments from the case $k_1, k_2, k_3, k_4>0$ with $(\omega_3,8)$ replaced by $(\omega_1,8)$.

\paragraph{\textbf{Case} $\bf{k_4=0, k_1, k_2, k_3}>\textbf{0}$.}
$\newline$
  We can suppose $\bf{p_{10}=0}$, and $p_5=0$ or $p_8=0$ (see the case $k_1, k_2, k_3, k_4>0$). If $p_8=0$ and $p_1\neq0$, then take $\tau=(\omega_2,1, 5)$ hence we get $p_{5}=0$ or $p_1=0$. If $p_8=0$ and $p_1=0$, then take $\tau=(\omega_2,5)$ hence we obtain $p_{5}=0$. Thus we can assume $\bf{p_5=0}$.
  $\newline$
  If $p_2\neq0$ and $p_6\neq0$, then take $\tau=(\omega_2,2,6)$ hence we get $p_{2}=0$ or $p_6=0$. We want to show that it is enough to assume $p_6=0$.
  $\newline$
  If $p_2=0$, $p_6\neq0$, then take $\tau=(\omega_1,6)$ hence we obtain $p_{6}=0$ or the inequality 16 turns into equality for $\sigma$.
  $\newline$
   If $p_8=0$, then take $\tau=(\omega_2,6)$ hence we can suppose $p_{6}=0$.
   $\newline$
   If $p_6\neq0$ and $p_8\neq0$, then take $\tau=(\omega_3,8)$ hence we get $p_{8}=0$ or the inequality 14 turns into equality for $\sigma$.
   $\newline$
   If the inequalities 14 and 16 turn into equalities together for $\sigma$, then take $\tau=(\omega_2,6,8)$ hence we obtain $p_{6}=0$ or $p_8=0$. Thus $\bf{p_6=0}$ in any case.
    $\newline$
    If $p_2\neq0$, then take $\tau=(\omega_3,2)$ hence we get $\bf{p_{2}=0}$. Thus we obtain that only coordinates $p_1, p_3, p_8, p_{11}$ can be nonzero. Now it easy to see that $\sigma$ is representable as a sum of the essential signatures of fundamental highest weights (essential inequalities are 2, 6, 16, 21).

\paragraph{\textbf{Case} $\bf{k_2=0, k_1, k_3, k_4}>\textbf{0}$.}
 $\newline$
 And again we can suppose that $\bf{p_{10}=0}$.
 $\newline$
  If $p_5\neq0$, then take $\tau=(\omega_5,5)$ hence we get $\bf{p_{5}=0}$. Then we can use the same arguments as in the case above with one modification: if $p_8\neq0$, then $\tau=(\omega_3,8)$ is suitable in any case (if inequality 14 turns into equality, then the inequality 20 is violated).

\paragraph{\textbf{Case} $\bf{k_1=0,k_2,k_3,k_4>0}$.}
$\newline$
See the case $k_1,k_2,k_3,k_4>0$.

\paragraph{\textbf{Case} $\bf{k_2, k_3}>\bf{0, k_1=0, k_4=0}$.}
$\newline$
We can assume $\bf{p_{10}=0}$, $\bf{p_5=0}$, and $p_2=0$ or $p_6=0$ (see the case $k_1, k_2, k_3>0, k_4=0$). If $p_6=0$, then see the case $k_1, k_2, k_3>0, k_4=0$.
If $\bf{p_2=0}$ and $p_3\neq0$, then take $\tau=(\omega_3,3)$ hence we get $\bf{p_{3}=0}$.
 $\newline$
 If $p_8\neq0$, then take $\tau=(\omega_3,8)$ hence we obtain $\bf{p_{8}=0}$. Thus we obtain that only $p_1, p_6, p_{11}$ can be nonzero. It is easy to see that $\sigma$ is representable as a sum of the essential signatures of fundamental highest weights (essential inequalities are 2, 9, 20).

\paragraph{\textbf{Case} $\bf{k_2, k_4}>\bf{0, k_1=0, k_3=0}$.}
 $\newline$
 We can assume $\bf{p_{10}=0}$, $\bf{p_5=0}$ and $p_2=0$ or $p_6=0$ (see the case $k_3=0, k_1, k_2, k_4>0$). If $p_2=0$ and $p_3\neq0$, then take $\tau=(\omega_4,3)$ hence we get $\bf{p_{3}=0}$. Next if $p_6\neq0$, then take $\tau=(\omega_4,6)$ hence  we can assume $\bf{p_{6}=0}$.
 $\newline$
 If $p_6=0$, then see the case $k_1, k_2, k_3, k_4>0$.

\paragraph{\textbf{Case} $\bf{k_1, k_2}>\bf{0, k_3=0, k_4=0}$.}
$\newline$
We have ${p_{3}=0}$ or $p_{10}=0$. We want to show that $p_{10}=0$ is the only case needed to be considered.
$\newline$
  If $p_3=0$ and $p_6\neq0$, then take $\tau=(\omega_2,6,{10})$ hence we get ${p_{6}=0}$ or $p_{10}=0$.
  If $p_3=0, p_6=0$ and $p_1\neq0$, then take $\tau=(\omega_2,1,{10})$ hence we obtain ${p_{1}=0}$ or $p_{10}=0$.
   $\newline$
   If $p_3=0, p_6=0, p_1=0$, then take $\tau=(\omega_2,{10})$ hence we get ${p_{10}=0}$. Thus in any case we can assume $\bf{p_{10}=0}$.
   $\newline$
   The arguments from the case $k_4=0, k_1, k_2, k_3>0$ show that we can suppose $\bf{p_5=0}$ and $p_2=0$ or $p_6=0$.
$\newline$
If $p_2=0$ and $p_3\neq0$, then take $\tau=(\omega_1,{3})$ hence we get $\bf{p_{3}=0}$. Thus only coordinates  $p_1, p_6, p_8, p_{11}$ can be nonzero. And $\sigma$ is representable as a sum of the essential signatures of fundamental highest weights (essential inequalities are 2, 6, 9, 43).
$\newline$
If $p_6=0$ and $p_3\neq0$, then take $\tau=(\omega_1,{3})$ hence we obtain $\bf{{p_{3}=0}}$. Thus only coordinates  $p_1, p_2, p_8, p_{11}$ can be nonzero. And $\sigma$ is representable as a sum of the essential signatures of fundamental highest weights (essential inequalities are 2, 12, 16, 43).
$\newline$
The remaining cases:
\begin{enumerate}
\item $k_1, k_3>0, k_2=k_4=0$
\item $k_1, k_4>0, k_2=k_3=0$
\item $k_3, k_4>0, k_2=k_1=0$
\item $k_2>0, k_1=k_3=k_4=0$
\item $k_1>0, k_2=k_3=k_4=0$
\item $k_3>0, k_1=k_2=k_4=0$
\item $k_4>0, k_1=k_3=k_2=0$.
\end{enumerate}
All these cases are trivial and can be partially reduced to the considered ones. For example, let us consider the case 4.
$\newline$
We can assume $\bf{p_{12}=0}, \bf{p_4=0}, \bf{p_9=0}, \bf{p_7=0}, \bf{p_{10}=0}, \bf{p_5=0}$, and $p_2=0$ or $p_6=0$ (see the case $k_1, k_2>0, k_3=0, k_4=0$). If $p_6=0$, then only the coordinates $p_1, p_2, p_3, p_8, p_{11}$ can be nonzero. Hence $\sigma$ is representable as a sum of essential signatures of fundamental highest weights (essential inequalities are 2, 16, 21).
 $\newline$
  If $p_2=0$, then only the coordinates $p_1, p_3, p_6, p_8, p_{11}$ can be nonzero. Hence $\sigma$ is representable as a sum of essential signatures of fundamental highest weights (essential inequalities are 2, 14, 16, 20).

   All the remaining cases can be considered by using similar arguments.

\end{proof}

\begin{proof}[Proof of Theorem \ref{t2}]
Let $\lambda=k_1\omega_1+k_2\omega_2+k_3\omega_3+k_4\omega_4$.
We need to show that $\dim V(\lambda)=|\Sigma^{f}(\lambda)|$ for all dominant weights $\lambda$. By (\cite[Sec. 5.4]{[T]}) there exists  a 4-variate polynomial $f(x_1,x_2,x_3,x_4)$ of total degree $\leq12$ such that
 $$
 f(k_1,k_2,k_3,k_4)= |\Sigma^{f}(\lambda)|,
 $$
 for an arbitrary dominant weight $\lambda$.
 By Weyl's dimension formula there exists another 4-variate polynomial $w(x_1,x_2,x_3,x_4)$ of degree 12 such that
 $$
 w(k_1,k_2,k_3,k_4)=\dim V(\lambda),
 $$
 for an arbitrary dominant weight $\lambda$.

 To prove that $w(x_1,x_2,x_3,x_4)=f(x_1,x_2,x_3,x_4)$ it is enough to verify this for all quadruples of non-negative integers $m_1,m_2,m_3,m_4$ such that $m_1+m_2+m_3+m_4\leq12$ (\cite[Sec. 5.4]{[T]}).
 This can be easily done by using a computer.
\end{proof}

\end{document}